\documentclass[10pt]{article}

\usepackage{cite, amsmath, amssymb, booktabs, lastpage, graphicx}
\usepackage[hidelinks]{hyperref}
\usepackage{placeins}
\usepackage{geometry}
\usepackage{setspace}
\usepackage{xurl} 

\usepackage[mathscr]{euscript}
\usepackage{amsthm}
\usepackage{xcolor}
\usepackage{tikz}
\usetikzlibrary{calc}
\theoremstyle{plain}

\newtheorem{theorem}{Theorem}

\newtheorem{proposition}[theorem]{Proposition}
\newtheorem*{corollary}{Corollary}

\theoremstyle{remark}
\newtheorem*{remark}{Remark}

\theoremstyle{definition}
\newtheorem{definition}{Definition}
\newtheorem{example}{Example}

\usepackage{graphicx}

\makeatletter
\renewcommand{\maketitle}{
\begin{center}

{\Large\bfseries \@title\par}
\vspace{6mm}

{\large\bfseries \@author\par}
\vspace{4mm}

{\itshape \@address\par}
\vspace{2mm}

{\small\ttfamily \@email\par}
\vspace{5mm}

\vspace{5mm}

\end{center}
}
\makeatother

\makeatletter
\newcommand{\address}[1]{\gdef\@address{#1}}
\newcommand{\email}[1]{\gdef\@email{#1}}
\makeatother

\address{}
\email{}

\usepackage{fancyhdr}

\usepackage[margin=1cm,%
			font=footnotesize,%
			format=hang,%
			labelsep=period,%
			labelfont=bf]{caption}
\allowdisplaybreaks
\title{On the Vertex Seidel Energy of Graphs}
\author{Kalpesh M. Popat$^{a}$, Enide Andrade$^b$}
\address{$^{a}$Department of Mathematics, \\ Saurashtra University,  Rajkot-360005, \\ Gujarat, India.\\[4mm]
$^{b}$Center for Research and Development in Mathematics and Applications (CIDMA),  Department of Mathematics, \\
University of Aveiro, 3810-193 Aveiro Portugal.}
\email{kalpeshmpopat@gmail.com, enide@ua.pt}
\date{\today}
\begin{document}
\maketitle
\thispagestyle{empty}
\begin{abstract}
We introduce the vertex Seidel energy via the diagonal entries of the absolute Seidel matrix. We establish a spectral formula, compute exact values for several graph families, derive bounds, and present a Coulson-type integral representation for analytical study of this invariant. We also show that vertex Seidel energy is invariant under Seidel switching and complementation.

\end{abstract}

\onehalfspacing

\section{Introduction}
Before presenting a brief state of the art, we introduce the notations and definitions used throughout the paper. Let $G$ be a simple undirected graph on $n$ vertices, and denote its adjacency matrix by $A(G)$. When vertices $v_i$ and $v_j$ are adjacent we just denote it by $v_i \sim v_j.$ The matrices $J_k$ and $I_k$ are the all-ones matrix and identity of orders $k$, respectively. When the order is understood from the context we just write $J$ and $I.$ The absolute value of a matrix $B$ is defined as $|B| = (BB^{\star})^{1/2}$ where $B^{\star}$ is its conjugate transpose. The transpose of $B$ is denoted by $B^\top$ and, when $B$ is square, the trace of $B$ is denoted by $\operatorname{tr}(B)$. We denote by $\mathbb{F}_q$ the finite field with $q$ elements.  We denote by $e_i \in \mathbb{R}^n$ the $i$-th standard basis vector.The inner product of two vectors $u,v \in \mathbb{R}^n$ is denoted by $\langle u,v \rangle.$ The complete graph and the complete bipartite graph are denoted by $K_n$ and $K_{m,n}$, respectively. Throughout the paper, $\mathrm{i}=\sqrt{-1}$ denotes the imaginary unit. We write $S=S(G)$ for the Seidel matrix of $G$ and $S^{(i)}$ for the principal submatrix obtained by deleting the $i$-th row and column (corresponding to vertex $v_i$). Let $\Phi(\lambda)$ denote the characteristic polynomial of $S$, and $\Phi_i(\lambda)$ the characteristic polynomial of $S^{(i)}$, that is, $\Phi(\lambda):=\det(\lambda I - S),\;
\Phi_i(\lambda):=\det(\lambda I - S^{(i)}).$ The \emph{Seidel spectrum} of $G$ is the eigenvalues of $S$ counted with multiplicity. When it is necessary to indicate the underlying graph explicitly, we write 
\(\mathcal{E}_S(v_i;G)\) for the vertex Seidel energy of \(v_i\) in \(G\).

A graph is called \emph{$k$-regular} if every vertex has degree $k$. A graph $G$ is said to be \emph{strongly regular} with parameters $(v,k,\lambda,\mu)$ if it has $v$ vertices, is $k$-regular, every pair of adjacent vertices has exactly $\lambda$ common neighbors, and every pair of non-adjacent vertices has exactly $\mu$ common neighbors. A \emph{conference graph} of order \(v\) is a strongly regular graph with parameters \(\Big(v,\;\tfrac{v-1}{2},\;\tfrac{v-5}{4},\;\tfrac{v-1}{4}\Big).\) For a prime power $q \equiv 1 \pmod{4}$, the \emph{Paley graph} $P_q$ has vertex set $\mathbb{F}_q$, the finite field with $q$ elements, where two distinct vertices $x$ and $y$ are adjacent if and only if $x-y$ is a nonzero quadratic residue in $\mathbb{F}_q$. Paley graphs are strongly regular with parameter \(v=q\). A graph is called \emph{vertex-transitive} if its automorphism group acts transitively on its vertex set. Let \(G\) be a simple graph. The \emph{complement} of \(G\), denoted by \(\overline{G}\), is the graph on the same vertex set in which two distinct vertices are adjacent in \(\overline{G}\) if and only if they are non-adjacent in \(G\). Let \(G\) be a graph on vertex set \(V\) and let \(X\subseteq V\). The \emph{Seidel switch} with respect to \(X\) produces the graph \(G^X\) whose edges between \(X\) and \(V\setminus X\) are complemented while edges inside \(X\) and inside \(V\setminus X\) remain unchanged. 

The interaction between spectral graph theory and mathematical chemistry gave rise to many ``energy'' invariants. The classical \emph{graph energy} of Gutman \cite{gutman1978energy} is the sum of absolute values of the eigenvalues of the adjacency matrix $A(G)$. The concept has since been systematized and extended in several directions; see the monographs \cite{li2012graph,brouwer2012,cvetkovic2010}.

The \emph{Seidel matrix} of a graph $G$ with vertex set $\{v_1,\dots,v_n\}$ is the $n\times n$ matrix $S(G)=[s_{ij}]$ defined by $s_{ii}=0$ for all $i$, and for $i\ne j$ by $s_{ij}=-1$ if $v_i\sim v_j$ and $s_{ij}=1$ otherwise. Equivalently, $S(G)=J-I-2A(G).$ If $\theta_1,\dots,\theta_n$ are the eigenvalues of $S(G)$ then the \emph{Seidel energy} \cite{haemers2012switching} is defined as
\[
\mathcal{E}_S(G):=\sum_{j=1}^n |\theta_j|.
\]
This invariant has been extensively studied in the literature; in particular, it is preserved under Seidel switching and graph complementation \cite{haemers2012switching,ramane,oboudi2016,popat2019,oboudi2023}. Haemers conjectured (and Akbari \emph{et al.} later proved) that $\mathcal{E}_S(G)\ge 2n-2$ for every graph of order $n$, with equality for $K_n$ \cite{haemers2012switching,akbari2020}.

Let $M\in\mathbb{R}^{n\times n}$ be a real symmetric matrixwith spectral decomposition $M=U\operatorname{diag}(\lambda_1,\dots,\lambda_n)U^\top$, where $U=[u_{ij}]$ is orthonormal. Then $|M|=(M^2)^{1/2}=U\operatorname{diag}(|\lambda_1|,\dots,|\lambda_n|)U^\top$. For $1\le i\le n$ denote by $\varphi_i$ the $i$-th coordinate functional $\varphi_i(X)=e_i^\top X e_i$. The \emph{vertex $M$-energy} at the coordinate $i$ (or at vertex $v_i$ when $M$ is a graph matrix) is
\[
\mathcal{E}_M(v_i):=\varphi_i\big(|M|\big)=|M|_{ii}=\sum_{j=1}^n u_{ij}^2\,|\lambda_j|.
\]
The total $M$-energy equals $\operatorname{tr}|M|=\sum_{j=1}^n|\lambda_j|$ and satisfies $\mathcal{E}(M)=\sum_{i=1}^n\mathcal{E}_M(v_i)$, so the vertex energies partition the total energy. This framework has been used extensively in the study of vertex energies: for adjacency vertex energy in \cite{arizmendi2018energy, arizmendi2019coulson,gutman2025tutorial,shrikanth2025vertex,nagesh2025distribution,dede2025double}, for Laplacian vertex energy in \cite{guerrero2022laplacian}, and for Randi\'c vertex energy in \cite{gogoi2025}.
 
Specializing the matrix-energy framework to the Seidel matrix $S(G)$, we introduce the \emph{vertex Seidel energy} of a vertex $v_i$; the formal definition, its spectral representation, and closed forms for standard families complete, complete bipartite and Paley graphs appear in Section~\ref{sec:seidel-vertex-energy}. We derive general bounds in Section~\ref{sec:bounds}, characterize when the vertex Seidel energy is constant in Section~\ref{ssec:constant-energy}, present a Coulson-type integral formula in Section~\ref{sec:coulson}, and discuss Seidel switching and complementation in Section~\ref{sec:switching-complement}.

\section{Vertex Seidel Energy}
\label{sec:seidel-vertex-energy}
In this section we introduce the central object of the paper, record its basic structural properties, and compute it in several standard cases.

\begin{definition}
Let \(G\) be a simple graph with Seidel matrix \(S=S(G)\). For a vertex \(v_i\) the \emph{vertex Seidel energy} is
\[
\mathcal{E}_S(v_i):=\bigl|S(G)\bigr|_{ii},
\]
where \(|S|=(S^2)^{1/2}\). Equivalently, if \(S=W\operatorname{diag}(\theta_1,\dots,\theta_n)W^\top\) with orthonormal \(W=[w_{ij}]\), then
\[
\mathcal{E}_S(v_i)=\sum_{j=1}^n w_{ij}^2\,|\theta_j|.
\]
\end{definition} 
\noindent The next proposition collects immediate but useful structural facts.
\begin{proposition}
\label{prop:partition}
With notation as above the total Seidel energy decomposes as
\[
\mathcal{E}_S(G)=\operatorname{tr}|S|=\sum_{i=1}^n\mathcal{E}_S(v_i).
\]
\end{proposition}
\begin{proof}
This is the tautological identity \(\operatorname{tr}|S|=\sum_i |S|_{ii}\). Equivalently, summing the spectral representation \(\mathcal{E}_S(v_i)=\sum_j w_{ij}^2|\theta_j|\) over \(i\) and using \(\sum_{i}w_{ij}^2=1\) yields \(\sum_i\mathcal{E}_S(v_i)=\sum_j|\theta_j|=\mathcal{E}_S(G)\).
\end{proof}
\begin{proposition}
\label{prop:two-abs-eigs}
Let the Seidel spectrum of \(G\) has absolute values taking exactly two distinct values \(a>b\ge0\); equivalently the spectrum is contained in \(\{\pm a,\pm b\}\) (with multiplicities). Then for every vertex \(v_i\),
\begin{equation}
\label{eq:two-abs-formula}
\mathcal{E}_S(v_i) \;=\; b + (a-b)\,\frac{n-1-b^2}{a^2-b^2}.
\end{equation}
\end{proposition}
\begin{proof}
Let \(\Pi_a\) denote the spectral projector onto the eigenspaces of \(S\) with \(|\theta|=a\). Since \(S^2\) has only the eigenvalues \(a^2\) and \(b^2\), the projector \(\Pi_a\) is a polynomial in \(S^2\); one convenient form is
\[
\Pi_a=\frac{S^2-b^2 I}{a^2-b^2}.
\]
Define $p_{ij}:=w_{ij}^2,\; 1\le i,j\le n.$ Then for every vertex $v_i$,
\[
\mathcal{E}_S(v_i)=\sum_{j=1}^n w_{ij}^2\,|\theta_j|=\sum_{j} p_{ij}|\theta_j|
= a\sum_{j:|\theta_j|=a} p_{ij} + b\sum_{j:|\theta_j|=b} p_{ij}.
\]
Set $q_i=\sum_{j:|\theta_j|=a} p_{ij}= (\Pi_a)_{ii}$. Since $\sum_j p_{ij}=1$ we obtain
\[
\mathcal{E}_S(v_i)=a q_i + b(1-q_i) = b + (a-b)q_i.
\]
Using the projector expression,
\[
q_i=(\Pi_a)_{ii}=\frac{[S^2]_{ii}-b^2}{a^2-b^2},
\]
Finally note that $[S^2]_{ii}=\sum_j S_{ij}^2=n-1$ for every vertex $v_i$, 
\[
\mathcal{E}_S(v_i)
= b + (a-b)\,\frac{n-1-b^2}{a^2-b^2}.
\]
\end{proof}
\noindent The two-absolute-eigenvalue formula immediately gives a number of concrete corollaries.
\begin{corollary}
\label{cor:Kn}
For the complete graph \(K_n\) (with \(n\ge3\)) every vertex has
\[
\mathcal{E}_S(v_i)=\frac{2(n-1)}{n},\qquad i=1,\dots,n.
\]
\end{corollary}
\begin{proof}
The Seidel spectrum of the complete graph \(K_n\) is \(1-n\) and \(1\) (with multiplicity \(n-1\)). Thus the two distinct absolute eigenvalues are \(a=n-1\) and \(b=1\). Using  Proposition \ref{prop:two-abs-eigs} we get
\[
\mathcal{E}_S(v_i)
=1+(n-2)\frac{n-2}{(n-1)^2-1}= \frac{2(n-1)}{n},\qquad i=1,\dots,n.
\]

\end{proof}
\begin{corollary}
\label{cor:Krs}
For complete bipartite graph \(G=K_{r,s}\) every vertex has
\[
\mathcal{E}_S(v_i)=\frac{2(r+s-1)}{r+s},\qquad i=1,\dots,n.
\]
\end{corollary}
\begin{proof}
Let \(p=r+s\). The Seidel spectrum of \(K_{r,s}\) is \(p-1\) and \(-1\) (with multiplicity \(p-1\). Hence the two absolute eigenvalues are \(a=p-1\) and \(b=1\). Applying Proposition \ref{prop:two-abs-eigs} gives
\[
\mathcal{E}_S(v_i)
=1+(p-2)\frac{p-2}{(p-1)^2-1}
=\frac{2(p-1)}{p}=\frac{2(r+s-1)}{r+s},\qquad i=1,\dots,n.
\]
\end{proof}
\begin{corollary}
\label{cor:conference-compact}
Let $G$ be a conference graph of order $v$. Then every vertex satisfies
\[
\mathcal{E}_S(v_i)=\frac{v-1}{\sqrt{v}}\qquad i=1,\dots,n.
\]
\end{corollary}

\begin{proof}
For a conference graph the adjacency spectrum is
\[
k=\frac{v-1}{2},\qquad
\lambda_{\pm}=\frac{-1\pm\sqrt{v}}{2}\quad(\text{each of multiplicity }(v-1)/2).
\]
Since $S=J-I-2A$, the all-ones vector $\mathbf{1}$ is an eigenvector of $S$ corresponding to the eigenvalue $0$ as:
\[
S\mathbf{1}=(v-1-2k)\mathbf{1}=0,
\]
and any eigenvector of $A$ orthogonal to $\mathbf{1}$ corresponding to the eigenvalue $\lambda$ yields an eigenvector of $S$ corresponding to the eigenvalue $-1-2\lambda$. Applying this to $\lambda=\lambda_{\pm}$ gives the Seidel spectrum $\mp\sqrt{v}$, each with multiplicity $(v-1)/2$, together with the simple eigenvalue $0$. 
Taking $a=\sqrt{v}$ and $b=0$ in Proposition~\ref{prop:two-abs-eigs} yields
\[
\mathcal{E}_S(v_i)
= 0 + (a-0)\,\frac{v-1-0}{a^2-0}
= \frac{v-1}{\sqrt{v}},\qquad i=1,\dots,n.
\]
\noindent (The same formula applies to Paley graphs \(P_q\) for \(q\equiv1\pmod4\).)
\end{proof}

\section{Bounds for Vertex Seidel Energy}
\label{sec:bounds}
In this section we establish several fundamental bounds that constrain the possible values of the vertex Seidel energies.
\begin{theorem}
\label{thm:cs-upper}
Let \(G\) be a graph of order \(n\). Then for every vertex \(v_i\),
\[
\mathcal{E}_S(v_i)\le\sqrt{n-1}.
\]
Moreover, equality holds for every vertex \(v_i\) if and only if
\[
S^2=(n-1)I_n.
\]
\end{theorem}
\begin{proof}
Set \(M:=|S|=(S^2)^{1/2}\). The matrix \(M\) is symmetric and positive semidefinite, so for each standard basis vector \(e_i\) we may apply the Cauchy--Schwarz inequality in \(\mathbb R^n\) to the vectors \(Me_i\) and \(e_i\):
\[
\bigl\langle Me_i,e_i\bigr\rangle^2 \le \langle Me_i,Me_i\rangle\langle e_i,e_i\rangle.
\]
The left-hand side equals \(\varphi_i(M)^2=\bigl(|S|\bigr)_{ii}^2=\mathcal{E}_S(v_i)^2\). The right-hand side equals \(e_i^\top M^2 e_i \cdot 1 = e_i^\top S^2 e_i = [S^2]_{ii}=n-1\). Therefore
\[
\mathcal{E}_S(v_i)^2 \le n-1.
\]
Suppose equality holds for every \(i\). Then for each \(i\) we have \(\mathcal{E}_S(v_i)^2 = [S^2]_{ii}\), i.e. $\langle Me_i,e_i\rangle^2 = \langle Me_i,Me_i\rangle\langle e_i,e_i\rangle.$
By the equality condition in Cauchy--Schwarz, the vectors \(Me_i\) and \(e_i\) are linearly dependent; hence there exists \(\lambda_i\in\mathbb R\) with \(Me_i=\lambda_i e_i\). Thus \(M\) is diagonal in the standard basis and \(M=\operatorname{diag}(\lambda_1,\dots,\lambda_n)\). Then
\[
[S^2]_{ii} = e_i^\top M^2 e_i = \lambda_i^2 = n-1,
\]
so \(\lambda_i=\sqrt{n-1}\) (since \(M\) is positive semidefinite). Therefore \(M=\sqrt{n-1}\,I_n\), equivalently \(|S|=\sqrt{n-1}\,I_n\). Squaring gives \(S^2=(n-1)I_n\). Conversely, if \(S^2=(n-1)I_n\) then \(|S|^2=S^2=(n-1)I_n\), whence \(|S|=\sqrt{n-1}\,I_n\) and \(\mathcal{E}_S(v_i)=\sqrt{n-1}\) for every \(i\). 
\end{proof}
\begin{remark}
The algebraic equality condition is \(S^2=(n-1)I\). Any graph whose Seidel matrix satisfies this matrix equation attains \(\mathcal{E}_S(v_i)=\sqrt{n-1}\) at every vertex. Such Seidel matrices are precisely symmetric \((0,\pm1)\)-matrices with zero diagonal and squared equal to \((n-1)I\); they are often called \emph{symmetric conference matrices} in the literature \cite{seidel1976twographs}. We do not assert existence for all \(n\) here; instead we note that when such matrices exist they provide sharpness for Theorem \ref{thm:cs-upper}.
\end{remark}

\paragraph{Equality examples.}
The two small graphs below illustrate the equality case \(S^2=(n-1)I_n\). In each example the displayed Seidel matrix satisfies \(S^2=(n-1)I_n\), hence \(|S|=\sqrt{n-1}\,I_n\) and every vertex attains the per-vertex Seidel energy \(\sqrt{n-1}\). The numerical vertex energies are shown inside the nodes.
\FloatBarrier
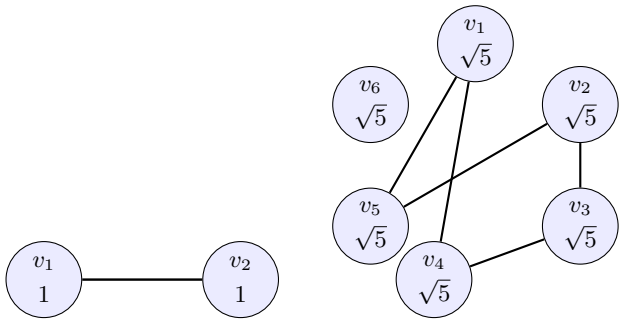
\begin{figure}[htbp]
  \centering
  \renewcommand{\arraystretch}{1.2}
  \begin{tabular}{c@{\qquad}c}
    % -------- K2 --------
    \begin{tikzpicture}[scale=1, every node/.style={font=\small}]
      % styling
      \tikzset{vtx/.style={circle,draw,fill=blue!8,minimum size=10mm,inner sep=0pt,align=center}}
      % nodes
      \node[vtx] (a) at (0,0) {$v_1$ \\[2pt] \(\;1\;\)};
      \node[vtx] (b) at (2.6,0) {$v_2$ \\[2pt] \(\;1\;\)};
      % edge
      \draw[line width=0.9pt] (a) -- (b);
    \end{tikzpicture}
    &
    % -------- Order-6 example --------
    \begin{tikzpicture}[scale=1, every node/.style={font=\small}]
      \tikzset{vtx/.style={circle,draw,fill=blue!8,minimum size=10mm,inner sep=0pt,align=center}}
      % place six vertices on a circle
      \foreach \i/\ang/\name in {
        1/90/v_1,
        2/30/v_2,
        3/330/v_3,
        4/250/v_4,
        5/210/v_5,
        6/150/v_6}
      {
        \node[vtx] (\name) at (\ang:1.6cm) {\(\name\)\\[2pt] \(\sqrt{5}\)};
      }
      % example edge set (one symmetric-realization achieving S^2=5I)
      \draw[line width=0.8pt] (v_1) -- (v_4);
      \draw[line width=0.8pt] (v_1) -- (v_5);
      \draw[line width=0.8pt] (v_2) -- (v_3);
      \draw[line width=0.8pt] (v_2) -- (v_5);
      \draw[line width=0.8pt] (v_3) -- (v_4);
      % optional light circle for layout
      % \draw[dashed,gray] (0,0) circle (1.6cm);
    \end{tikzpicture}
  \end{tabular}
  \caption{Left: \(K_2\) with vertex Seidel energy \(\sqrt{1}=1\). 
  Right: Conference type example of order \(6\) with vertex  Seidel energy \(\sqrt{5}\).}
  \label{fig:examples}
\end{figure}
\FloatBarrier
\begin{theorem}
\label{thm:holder-lower}
Let \(G\) be a graph of order \(n\ge2\). Then for every vertex \(v_{i}\),
\[
\mathcal{E}_S(v_i)\;\ge\;\frac{(n-1)^{3/2}}{\sqrt{[S^4]_{ii}}}.
\]
\end{theorem}
\begin{proof}
Let \(\theta_1,\dots,\theta_n\) be the Seidel eigenvalues and choose an orthonormal eigenbasis so that \(S=W\operatorname{diag}(\theta_1,\dots,\theta_n)W^\top\) with \(W=[w_{ij}]\). Set $a_j=|\theta_j|\ge0,\; p_j=w_{ij}^2\ge0,$ so that \(\sum_{j=1}^n p_j=1\). Then
\[
\mathcal{E}_S(v_i)=\sum_{j=1}^n p_j a_j,\qquad
[S^2]_{ii}=\sum_{j=1}^n p_j a_j^2,\qquad
[S^4]_{ii}=\sum_{j=1}^n p_j a_j^4.
\]
Observe the identity
\[
\sum_{j=1}^n p_j a_j^2 \;=\; \sum_{j=1}^n (p_j a_j)^{2/3}\,(p_j a_j^4)^{1/3}.
\]
Apply Hölder's inequality with exponents \(p=3/2\) and \(q=3\) to the nonnegative sequences \(x_j:=(p_j a_j)^{2/3}\) and \(y_j:=(p_j a_j^4)^{1/3}\). Hölder gives
\[
\sum_{j=1}^n x_j y_j
\le\Big(\sum_{j=1}^n x_j^{3/2}\Big)^{2/3}\Big(\sum_{j=1}^n y_j^3\Big)^{1/3}.
\]
Computing the sums,
\[
\sum_{j=1}^n x_j^{3/2}=\sum_{j=1}^n (p_j a_j)^{(2/3)\cdot(3/2)}=\sum_{j=1}^n p_j a_j=\mathcal{E}_S(v_i),
\]
\[
\sum_{j=1}^n y_j^3=\sum_{j=1}^n (p_j a_j^4)^{(1/3)\cdot 3}=\sum_{j=1}^n p_j a_j^4=[S^4]_{ii}.
\]
Therefore
\[
\sum_{j=1}^n p_j a_j^2
\le \mathcal{E}_S(v_i)^{2/3}\,[S^4]_{ii}^{1/3}.
\]
Rearranging yields
\[
\mathcal{E}_S(v_i)\ge\frac{\big(\sum_{j=1}^n p_j a_j^2\big)^{3/2}}{\big(\sum_{j=1}^n p_j a_j^4\big)^{1/2}}
= \frac{[S^2]_{ii}^{3/2}}{\sqrt{[S^4]_{ii}}}=\frac{(n-1)^{3/2}}{\sqrt{[S^4]_{ii}}}
\]
\end{proof}

\section{When is the vertex Seidel energy constant?}
\label{ssec:constant-energy}
The concrete families discussed above—complete, complete bipartite, and Paley/conference graphs—share the feature that every vertex has the same Seidel energy. This prompts the natural question: which structural properties force this constancy? Crucially, regularity alone does not.

For example, the non-regular conference-type graph of order~6 in Figure~\ref{thm:holder-lower} has identical vertex Seidel energies despite unequal degrees. Conversely, a small perturbation of the Petersen graph yields a connected $3$-regular graph whose vertex Seidel energies are not all equal: here the degree sequence is uniform but the eigenvector-weight distribution is not, so the diagonal entries of $|S|$ differ. These examples show that constancy is a spectral phenomenon rather than a mere consequence of degree regularity.
\FloatBarrier
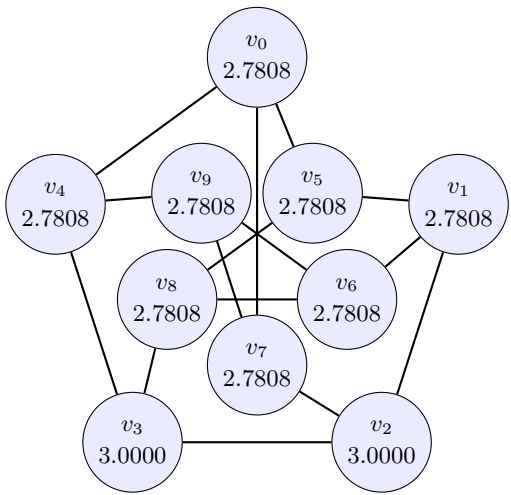
\begin{figure}[htbp]
  \centering
  \begin{tikzpicture}[scale=1, every node/.style={font=\small}]
    % node style (compact, allows line breaks)
    \tikzset{vtx/.style={circle,draw,fill=blue!8,minimum size=9mm,inner sep=1pt,align=center}}

    % radii chosen to avoid overlap (tested spacing)
    \def\rout{2.8cm}   % outer pentagon radius
    \def\rin{1.25cm}   % inner star radius

    % Outer pentagon vertices (angles in degrees, 72-degree spacing)
    \node[vtx] (v0) at (90:\rout)  {\(v_0\)\\[1pt]\(\;2.7808\;\)};
    \node[vtx] (v1) at (18:\rout)  {\(v_1\)\\[1pt]\(\;2.7808\;\)};
    \node[vtx] (v2) at (-54:\rout) {\(v_2\)\\[1pt]\(\;3.0000\;\)};
    \node[vtx] (v3) at (-126:\rout){\(v_3\)\\[1pt]\(\;3.0000\;\)};
    \node[vtx] (v4) at (-198:\rout){\(v_4\)\\[1pt]\(\;2.7808\;\)};

    % Inner star vertices (angles offset to sit between outer vertices)
    \node[vtx] (v5) at (54:\rin)   {\(v_5\)\\[1pt]\(\;2.7808\;\)};
    \node[vtx] (v6) at (-18:\rin)  {\(v_6\)\\[1pt]\(\;2.7808\;\)};
    \node[vtx] (v7) at (-90:\rin)  {\(v_7\)\\[1pt]\(\;2.7808\;\)};
    \node[vtx] (v8) at (-162:\rin) {\(v_8\)\\[1pt]\(\;2.7808\;\)};
    \node[vtx] (v9) at (-234:\rin) {\(v_9\)\\[1pt]\(\;2.7808\;\)};

    % --- edges for the modified Petersen (2-edge swap) ---

    % Outer cycle edges: omit (v0)-(v1) as per modification
    \draw[line width=0.8pt] (v1) -- (v2);
    \draw[line width=0.8pt] (v2) -- (v3);
    \draw[line width=0.8pt] (v3) -- (v4);
    \draw[line width=0.8pt] (v4) -- (v0);

    % Inner star edges: originally (v5-v7),(v7-v9),(v9-v6),(v6-v8),(v8-v5)
    % remove (v5)-(v7) per modification
    \draw[line width=0.8pt] (v7) -- (v9);
    \draw[line width=0.8pt] (v9) -- (v6);
    \draw[line width=0.8pt] (v6) -- (v8);
    \draw[line width=0.8pt] (v8) -- (v5);

    % Spokes (outer to corresponding inner)
    \draw[line width=0.8pt] (v0) -- (v5);
    \draw[line width=0.8pt] (v1) -- (v6);
    \draw[line width=0.8pt] (v2) -- (v7);
    \draw[line width=0.8pt] (v3) -- (v8);
    \draw[line width=0.8pt] (v4) -- (v9);

    % Added edges (the 2-edge swap)
    \draw[line width=0.8pt] (v0) -- (v7); % added
    \draw[line width=0.8pt] (v1) -- (v5); % added
  \end{tikzpicture}
  \caption{A connected 3-regular graph (modified Petersen) with non-constant vertex Seidel energy.}
  \label{fig:modified-petersen-compact}
\end{figure}
\FloatBarrier
Below are three (non-exhaustive) structural mechanisms that force constant vertex Seidel energy. A full characterization of graphs with constant vertex Seidel energy remains an open problem.
\smallskip
\noindent
(i) If $S^2=\alpha I$ for some $\alpha>0$, then $|S|=\sqrt{\alpha}\,I$ and $\mathcal{E}_S(v_i)=\sqrt{\alpha}$ for every vertex $v_i$ (conference graphs are the primary examples).
\smallskip
\noindent
(ii) If the Seidel spectrum has only two distinct absolute values, then each vertex Seidel energy is the same convex combination of those values; this covers complete, complete bipartite, and Paley graphs.

\smallskip
\noindent
(iii) If $G$ is vertex-transitive, then $|S|$ has a constant diagonal, hence all vertices have the same Seidel energy.
\section{Coulson-type integral representation}
\label{sec:coulson}
In this section we derive a Coulson–type integral representation for the vertex Seidel energy (written in terms of the characteristic polynomial of the Seidel matrix and of its principal minors). Such formulas originate in chemical graph theory, where Coulson’s integral formula expresses the graph energy in terms of its characteristic polynomial.

\begin{theorem}
\label{thm:coulson}
Let \(G\) be a graph with Seidel matrix \(S\). For any vertex \(v_i\),
\[
\mathcal{E}_S(v_i)
=\frac{1}{\pi}\,\mathrm{p.v.}\int_{-\infty}^{\infty}
\Bigg(1 - \mathrm{i} t\;\frac{\Phi_i(\mathrm{i} t)}{\Phi(\mathrm{i} t)}\Bigg)\,dt,
\]
where ``p.v.'' denotes the Cauchy principal value.
\end{theorem}
\begin{proof}
Let \(\theta_1,\dots,\theta_n\) be the (real) eigenvalues of \(S\) and let \(\{w_k\}_{k=1}^n\) be an orthonormal basis of eigenvectors with \(w_k=(w_{k,1},\dots,w_{k,n})^\top\). Then
\[
\mathcal{E}_S(v_i)=\sum_{k=1}^n |\theta_k|\, (w_{k,i})^2.
\]
We first show that the scalar identity
\[
|\theta|=\frac{1}{\pi}\,\mathrm{p.v.}\int_{-\infty}^{\infty}\Big(1 - \frac{\mathrm{i} t}{\mathrm{i} t-\theta}\Big)\,dt
\]
holds for every real \(\theta\). If \(\theta=0\) the integrand vanishes identically and both sides are zero, so assume \(\theta\neq0\). Write
\[
1-\frac{i t}{i t-\theta}
=\frac{i t-\theta-i t}{i t-\theta}
=\frac{-\theta}{i t-\theta}
=\frac{\theta}{\theta-i t}.
\]
Hence
\[
\mathrm{p.v.}\int_{-\infty}^{\infty}\Big(1-\frac{i t}{i t-\theta}\Big)\,dt
=\theta\;\mathrm{p.v.}\int_{-\infty}^{\infty}\frac{dt}{\theta-i t}.
\]
Now
\[
\frac{1}{\theta-i t}=\frac{\theta+i t}{\theta^2+t^2}
=\frac{\theta}{\theta^2+t^2}+i\,\frac{t}{\theta^2+t^2}.
\]
The second term is an odd function of \(t\), so its principal value integral over \(\mathbb{R}\) vanishes. Therefore
\[
\mathrm{p.v.}\int_{-\infty}^{\infty}\frac{dt}{\theta-i t}
=\theta\int_{-\infty}^{\infty}\frac{dt}{\theta^2+t^2}
=\theta\cdot\frac{\pi}{|\theta|}=\pi\,\frac{\theta}{|\theta|}.
\]
Combining the two displayed formulas yields
\[
\mathrm{p.v.}\int_{-\infty}^{\infty}\Big(1-\frac{i t}{i t-\theta}\Big)\,dt
=\theta\cdot\pi\,\frac{\theta}{|\theta|}=\pi|\theta|,
\]
and dividing by \(\pi\) gives the identity.
 \\
Multiplying this identity by $(w_{k,i})^2$ and summing over $k$ gives
\[
\mathcal{E}_S(v_i)
=\frac{1}{\pi}\,\mathrm{p.v.}\int_{-\infty}^{\infty}
\Bigg(1 - \mathrm{i} t\sum_{k=1}^n\frac{(w_{k,i})^2}{\mathrm{i} t-\theta_k}\Bigg)\,dt,
\]
Recognizing the finite spectral expansion of the resolvent
\[
(\mathrm{i} t I - S)^{-1}=\sum_{k=1}^n\frac{1}{\mathrm{i} t-\theta_k}\, w_k w_k^\top,
\]
we see that
\[
\sum_{k=1}^n\frac{(w_{k,i})^2}{\mathrm{i} t-\theta_k}=\bigl((\mathrm{i} t I - S)^{-1}\bigr)_{ii}.
\]
Finally, the adjugate/determinant identity yields
\[
\bigl((\mathrm{i} t I - S)^{-1}\bigr)_{ii}=\frac{\det(\mathrm{i} t I - S^{(i)})}{\det(\mathrm{i} t I - S)}
=\frac{\Phi_i(\mathrm{i} t)}{\Phi(\mathrm{i} t)}.
\]
Substituting this into the integral expression above gives
\[
\mathcal{E}_S(v_i)
=\frac{1}{\pi}\,\mathrm{p.v.}\int_{-\infty}^{\infty}
\Big(1 - \mathrm{i} t\,\frac{\Phi_i(\mathrm{i} t)}{\Phi(\mathrm{i} t)}\Big)\,dt.
\]
\end{proof}
\begin{example}
Let \(G=K_n\). Then \(S=I-J\) and, with the usual substitution \(a=\mathrm{i} t-1\), $\mathrm{i} t I - S = a I + J.$ By the Sherman--Morrison formula,
\[
(aI+J)^{-1}=\frac{1}{a}I-\frac{1}{a(a+n)}J,
\]
so the diagonal resolvent entry equals
\[
\bigl((\mathrm{i} t I - S)^{-1}\bigr)_{ii}
=\frac{1}{a}-\frac{1}{a(a+n)}
=\frac{a+n-1}{a(a+n)}.
\]
Hence the integrand appearing in Theorem~\ref{thm:coulson} is
\[
1 - \mathrm{i} t\;\frac{a+n-1}{a(a+n)}\;=\;1 - \mathrm{i} t\;\frac{\mathrm{i}t-1+n-1}{(\mathrm{i}t-1)(\mathrm{i}t-1+n)}\;=\;\frac{n-1}{(t+\mathrm{i})\bigl(t-\mathrm{i}(n-1)\bigr)}.
\]
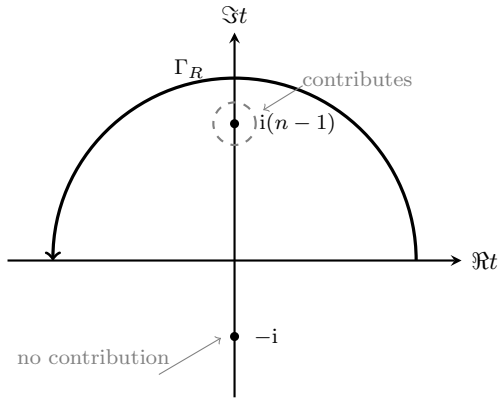
\begin{figure}[htbp]
  \centering
  \vspace{6pt}
  \begin{tikzpicture}[scale=1, every node/.style={font=\small}]
    % parameters
    \def\R{2.4}           % radius for semicircle
    \def\yup{1.8}         % y-position of upper pole (i(n-1) schematic)
    \def\ylow{-1.0}       % y-position of lower pole (-i)
    % styles
    \tikzset{
      axis/.style={->, thick, >=stealth},
      contour/.style={very thick},
      pole/.style={fill=black, circle, inner sep=1.2pt},
      labelsmall/.style={font=\footnotesize}
    }
    % axes
    \draw[axis] (-\R-0.6,0) -- (\R+0.6,0) node[right] {\(\Re t\)};
    \draw[axis] (0,-1.8) -- (0,\R+0.6) node[above] {\(\Im t\)};
    % real axis segment (with arrow indicating positive direction)
    %\draw[contour,->] (-\R,0) -- (\R,0) node[midway, below=3pt] {real axis};
    % semicircle in upper half-plane (arrow indicates orientation)
    \draw[contour,->] (\R,0) arc (0:180:\R);
    % poles (plotted off-axis to avoid overlap with contour)
    \node[pole] (low) at (0,\ylow) {};
    \node[right=2pt,labelsmall] at (low.east) {\( -\mathrm{i}\)};
    \node[pole] (up)  at (0,\yup) {};
    \node[right=4pt,labelsmall] at (up.east) {\( \mathrm{i}(n-1)\)};
    % small highlighted circle around contributing pole
    \draw[thick,dashed,gray] (0,\yup) circle (0.28);
    % arrow/annotation for contributing pole
    \draw[->,gray] (0.9,\yup+0.5) -- (0.38,\yup+0.18) node[midway, right=28pt,above=1pt,labelsmall] {contributes};
    % annotation for non-contributing pole
    \draw[->,gray] (-1.0,\ylow-0.5) -- (-0.18,\ylow+-0.02) node[midway, left=4  pt,labelsmall] {no contribution};
    % contour label
    \node[labelsmall,above] at (-0.6,0.95*\R) {\(\Gamma_R\) };
    % bounding box to avoid overlap with surrounding text
    \useasboundingbox (-\R-0.7,-1.9) rectangle (\R+0.7,\R+0.3);
  \end{tikzpicture}
  \vspace{3pt}
  \caption{Contour in the complex \(t\)-plane for the Coulson integral. The pole at \(t=\mathrm{i}(n-1)\) (upper half-plane) is enclosed and contributes via its residue; the pole at \(t=-\mathrm{i}\) lies in the lower half-plane and does not contribute.}
  \label{fig:contour}
\end{figure}
The integrand is a rational function with simple poles at \(t=-\mathrm{i}\) (lower half-plane) and \(t=\mathrm{i}(n-1)\) (upper half-plane). Closing the contour with a large semicircle in the upper half-plane picks up the residue at \(t=\mathrm{i}(n-1)\); the semicircle contribution vanishes because the integrand decays like \(O(t^{-2})\) as \(|t|\to\infty\).

\[
\mathrm{p.v.}\int_{-\infty}^{\infty}\frac{n-1}{(t+\mathrm{i})(t-\mathrm{i}(n-1))}\,dt
=2\pi \mathrm{i}\;\operatorname{Res}\!\Bigg(\frac{n-1}{(t+\mathrm{i})(t-\mathrm{i}(n-1))},\,t=\mathrm{i}(n-1)\Bigg).
\]
Since the pole at \(t=\mathrm{i}(n-1)\) is simple, its residue equals
\[
\operatorname{Res}\Big(\frac{n-1}{(t+\mathrm{i})(t-\mathrm{i}(n-1))},\,t=\mathrm{i}(n-1)\Big)
=\frac{n-1}{\,\mathrm{i}(n-1)+\mathrm{i}\,}=\frac{n-1}{\mathrm{i} n}=-\frac{\mathrm{i}(n-1)}{n}.
\]
Therefore the contour integral equals
\[
2\pi \mathrm{i}\cdot\Big(-\frac{\mathrm{i}(n-1)}{n}\Big)=\frac{2\pi(n-1)}{n},
\]
and dividing by \(\pi\) as in Theorem~\ref{thm:coulson} yields the per-vertex energy
\[
\mathcal{E}_S(v_i)=\frac{2(n-1)}{n}.
\]
\end{example}

\section{Seidel switching and complementation}
\label{sec:switching-complement}
Seidel switching and graph complementation act naturally on the Seidel matrix and preserve the spectrum of $|S|$, hence the vertex Seidel energies. In this section we prove that both operations leave the vertex Seidel energy invariant.

Let $G$ be a graph with vertex set $V$ and \(X\subseteq V\). Consider \(D_X:=\operatorname{diag}(d_1,\dots,d_n)\) be the diagonal matrix where \(d_i=-1\) when the vertex \(v_i\) belongs to \(X\), and \(d_i=1\) otherwise. Then the Seidel matrices of \(G\) and \(G^X\) satisfy
\[
S(G^X)=D_X\,S(G)\,D_X.
\]
\begin{theorem}
\label{thm:switch-complement-inv}
Let \(G\) be an \(n\)-vertex graph, \(G^X\) a Seidel switch of \(G\) with respect to \(X\). Then for every vertex \(v_i\)
\[
\mathcal{E}_S(v_i;G^X)=\mathcal{E}_S(v_i;G).
\]
\end{theorem}
\begin{proof}
Write \(S=S(G)\). For switching use \(D=D_X\). Since \(D\) is diagonal with entries \(\pm1\) we have \(D^{-1}=D\) and \(D^2=I\), and $S'= S(G^X) = D S D.$ Therefore
\[
S'^2 = D S^2 D,
\qquad |S'|=(S'^2)^{1/2}=D (S^2)^{1/2} D = D\,|S|\,D,
\]
where the middle equality follows from conjugation through the spectral decomposition. Taking diagonal entries gives
\[
|S'|_{ii} = e_i^\top D |S| D e_i = d_i^2\,|S|_{ii} = |S|_{ii},
\]
so \(\mathcal{E}_S(v_i;G^X)=|S'|_{ii}=\mathcal{E}_S(v_i;G)\) for every \(i\).
\end{proof}
\begin{theorem}
Let \(G\) be a graph and \(\overline G\) its complement. Then for every vertex \(v_i\) 
\[
\mathcal{E}_S(v_i;\overline G)=\mathcal{E}_S(v_i;G),
\]
\end{theorem}
\begin{proof}
Using \(S(G)=J-I-2A(G)\) and \(A(\overline G)=J-I-A(G)\) we obtain
\[
S(\overline G)=J-I-2A(\overline G)=J-I-2(J-I-A(G))=-S(G).
\]
Hence \(S(\overline G)\) is the negative of \(S(G)\). In particular, if
\[
S(G)=W\operatorname{diag}(\theta_1,\dots,\theta_n)W^\top
\]
with orthogonal \(W\), then
\[
S(\overline G)=W\operatorname{diag}(-\theta_1,\dots,-\theta_n)W^\top.
\]
Thus \(S(G)\) and \(S(\overline G)\) have the same eigenvectors, while their eigenvalues are negated. Consequently, the weight matrix \(P\) is identical for \(G\) and \(\overline G\), and
\[
|S(\overline G)|=|-S(G)|=|S(G)|.
\]
It follows that for every vertex \(v_i\),
\[
\mathcal{E}_S(v_i;\overline G)=\bigl|S(\overline G)\bigr|_{ii}
=\bigl|S(G)\bigr|_{ii}=\mathcal{E}_S(v_i;G).
\]
\end{proof}

\section{Concluding remark.}
The vertex Seidel energy provides a natural vertex-level refinement of Seidel energy, combining spectral, combinatorial, and analytic viewpoints. The results obtained here demonstrate that this invariant admits explicit formulas, meaningful bounds, and efficient analytical representations, opening several directions for further study.

\noindent\textbf{Acknowledgments.}
E.~Andrade's work is supported by CIDMA (\url{https://ror.org/05pm2mw36}) under the Portuguese Foundation for Science and Technology (FCT, \url{https://ror.org/00snfqn58}), through the grants
\href{https://doi.org/10.54499/UID/04106/2025}{UID/04106/2025} and
\href{https://doi.org/10.54499/UID/PRR/04106/2025}{UID/PRR/04106/2025}.

\end{document}